\documentclass[11pt]{article}

\usepackage[T1]{fontenc}
\usepackage[utf8]{inputenc}

\usepackage{amsmath,amssymb,amsfonts}
\usepackage{amsthm}
\usepackage{hyperref}
\usepackage{graphicx}
\usepackage{wrapfig}
\usepackage{multirow}
\usepackage{geometry}
\usepackage{babel}
\usepackage[all,cmtip]{xy}

\usepackage{listings}
\usepackage{xcolor}

\lstdefinelanguage{Julia}{
  keywords={function,return,for,in,if,else,end,using},
  sensitive=true,
  comment=[l]\#,
  morestring=[b]",
}

\lstdefinestyle{julia}{
  language=Julia,
  basicstyle=\ttfamily\small,
  keywordstyle=\mathcal{O}lor{blue},
  commentstyle=\mathcal{O}lor{gray},
  stringstyle=\mathcal{O}lor{red!70!black},
  numbers=left,
  numberstyle=\tiny\mathcal{O}lor{gray},
  stepnumber=1,
  numbersep=8pt,
  frame=none,
  backgroundcolor=\mathcal{O}lor{white},
  showstringspaces=false,
  breaklines=true,
  tabsize=2
}

\numberwithin{equation}{section}
\numberwithin{figure}{section}

\theoremstyle{plain}
\newtheorem{thm}{Theorem}[section]
\newtheorem{lem}[thm]{Lemma}
\newtheorem{prop}[thm]{Proposition}

\theoremstyle{definition}
\newtheorem{defn}[thm]{Definition}
\newtheorem{example}[thm]{Example}
\newtheorem{quest}[thm]{Question}

\theoremstyle{remark}
\newtheorem{rem}[thm]{Remark}

\newcommand{\ot}{\otimes}
\newcommand{\op}{\oplus}
\newcommand{\su}{\subseteq}

\newcommand{\ga}{\alpha}

\newcommand{\R}{\mathbb{R}}
\newcommand{\C}{\mathbb{C}}
\renewcommand{\P}{\mathbb{P}}
\newcommand{\Gr}{\mathrm{Gr}}
\newcommand{\im}{\mathrm{Im}}
\newcommand{\Sym}{\mathrm{Sym}}
\newcommand{\rk}{\mathrm{rk}}
\newcommand{\Hom}{\mathrm{Hom}}
\newcommand{\BW}{\mathrm{BW}}
\newcommand{\ED}{\mathrm{ED}}
\newcommand{\vED}{v\mathrm{ED}}
\newcommand{\Seg}{\mathrm{Seg}}

\title{On the Euclidean Distance Degree of Quadratic Two-Neuron Neural Networks}
\author{Giacomo Graziani}
\date{}

\begin{document}

\maketitle
\begin{abstract}
We study the Euclidean Distance degree of algebraic neural network models from the perspective of algebraic geometry. Focusing on shallow networks with two neurons, quadratic activation, and scalar output, we identify the associated neurovariety with the second secant variety of a quadratic Veronese embedding. We introduce and analyze the virtual Euclidean Distance degree, a projective invariant defined as the sum of the polar degrees of the variety, which coincides with the usual Euclidean Distance degree for a generic choice of scalar product. Using intersection theory, Chern–Mather classes, and the Nash blow-up provided by Kempf’s resolution, we reduce the computation of the virtual Euclidean Distance degree to explicit intersection numbers on a Grassmannian. Applying equivariant localization, we prove that this invariant depends stably polynomially on the input dimension. Numerical experiments based on homotopy continuation illustrate the dependence of the Euclidean Distance degree on the chosen metric and highlight the distinction between the generic and nongeneric cases, such as the Bombieri–Weyl metric.
\end{abstract}

\tableofcontents

\section{Introduction}

The algebraic geometry of neurovarieties has attracted increasing attention in recent years. This is due to the central role played by algebraic and semi-algebraic models in Deep Learning, supported by density and universal approximation results \cite{SM17,SM18}, by the existence of global invariants describing the expressive power of architectures \cite{KTB19}, by the relative tractability of the inevitable singularities these varieties exhibit — that are more delicate to work with in a differential-geometric setting \cite{Wat09} — and by the appearance of programmatic works such as \cite{MSMTK25}. A central invariant in Machine Learning is the $\ED$-degree, which measures the number of functions within a model class that optimally fit the training data. This invariant depends sensitively on the choice of a metric on the ambient space; however, for a sufficiently general scalar product its value stabilizes. This general value is shown to be a projective invariant and it is given by the sum of the polar classes of the variety, which we call the virtual $\ED$-degree of the variety $\vED$ (see \cite{DHOST} for a general treatment of this topic).
In this paper we study the virtual $\ED$-invariant, that for the neurovarieties associated to a shallow model with two neurons and scalar output and we use Intersection Theory to show that this value is a stably polynomial function on the dimension of the input space.
Closely related models were studied in \cite{KLW24}, where the authors allow the output dimension to vary and, by exploiting explicit formulas relating the virtual $\ED$-degree to intersection matrices on suitable Grassmannians, obtain closed expressions for $\vED$. 

The paper is organized as follows. After recalling some background material on rational resolutions used in the paper (Section \ref{sec:preliminaries}), we introduce the network under consideration, its building blocks, and the associated varieties (Section \ref{section:Definitions}). 
We then briefly review the norms and scalar products in Section \ref{sec:norms}, with particular emphasis on the Bombieri-Weyl product. In Sections \ref{sec:EDGeneral} and \ref{sec:vED} we apply our discussion to neurovarieties and explicit the relation between the $\ED$-degree and the virtual $\ED$-degree. The material up to this point is standard or well known. Section \ref{sec:Ourcase} is the technical core of the paper. In \ref{sec:polardegrees} we review the basic tools from Intersection theory that allow us to reduce the computation of the virtual $\ED$-degree to an integral on a Grassmannian variety. Finally in \ref{sec:localization} we use the natural action of the algebraic torus on the Grassmannian varieties to reduce the integrals to an  estimate of Edidin-Graham equivariant localization formula. Our statement follows from a direct inspection of the resulting expression.

In Section \ref{sec:numerical} we show an explicit example of the dependence of the $\ED$-degree on the chosen metric. The computations are performed making use of the technique of homotopy continuation implemented in \texttt{Julia}.

\section{Preliminaries}
\subsection{Conventions and notation}\label{sec:preliminaries}
All varieties are defined over the real numbers, while most of the geometric constructions (such as tangent spaces, conormal varieties, and $\ED$-correspondences) are performed on the complexification. Vector spaces are assumed to be finite dimensional. Given a vector space $V$ we denote with $\mathrm{Sym}^r V$ the space of symmetric tensors of order $r$ and with 
$v_r \left( V \right) \su \P\left( \mathrm{Sym}^r V \right)$ the image of the Veronese embedding $\left[ v \right] \mapsto \left[ v^{\ot r} \right]$. For any projective variety $X \subset \mathbb{P}^N$, we denote by $\sigma_k \left( X \right)$ its $k$-th secant variety, defined as the Zariski closure of the union of linear spans of $k$ points on $X$.
In this setting, the variety of symmetric tensors of rank at most $k$ is identified with $\sigma_k \left( v_r \left( V \right) \right)$. We refer to \cite{Lan12} for standard definitions and properties of tensor rank and secant varieties.

\subsection{Resolution of $\sigma_2\left( v_2 \left( V \right) \right)$}
\label{section:resolution}

We recall here some basic properties of rational resolutions of varieties over a field of characteristic 0 and recall the Kempf resolution of $\sigma_2\left( v_2 \left( V \right) \right)$. We refer to \cite{Wey} for further details.

\begin{defn}
Let $X$ be an irreducible projective variety. We say that a birational morphism $f:Z\to X$ is a resolution of singularities if $Z$ is a smooth projective variety. Moreover we say that $X$ has rational singularities if it is normal and there exists a resolution of singularities $f$ with $R^{i}f_{\ast}\mathcal{O}_{Z}=0$ for $i\ge1$.
\end{defn}

\begin{rem}\label{rem:RationalSingularities}
It can be shown that if $X$ has rational singularities using a resolution $f$, then any other resolution of singularities $g$ has the same property. \cite[Proposition 1.2.29]{Wey}
\end{rem}
Fix an $n$-dimensional vector space $V$ and let
\[
\mathcal M = \sigma_2 \left( v_2 \left( V \right) \right) \su
\P\left( \mathrm{Sym}^2 V \right).
\]
In this case it is possible to perform explicit computations using quadratic forms since we have a natural identification 
\[
\mathcal M=\P\left\{Q \in \mathrm{Sym}^2 V^\ast \,\big| \, \mathrm{rk}\left(Q \right)\le 2\right\}=\P\{ A\in\mathrm{Mat}^\mathrm{sym}_{n\times n}\left(\C\right)\,\big| \,\mathrm{rk}\left( A\right)\le2\}.
\]
\begin{rem}
The smooth points of $\mathcal M$ correspond to forms of rank $2$ and we compute the tangent space at those points: if $\C_\varepsilon=\C[\varepsilon]/ \left( \varepsilon^2 \right)$ denotes the ring of dual numbers and $A$ is a rank 2 symmetric form then
\[
T_A\mathcal M = \left\{ B\in\Sym^2 V^\ast\, \big| \, A+\varepsilon B\in \mathcal M\left(  \C_\varepsilon \right) \right\},
\]
that is, all $3\times3$ minors of $A+\varepsilon B$ vanish in $ \C_\varepsilon$. Since $A$ is symmetric of rank $2$, there exists a decomposition $V=U\oplus R$ in which
\[
A = 
\begin{pmatrix}
I_2 & 0\\
0 & 0
\end{pmatrix}
\qquad
B =
\begin{pmatrix}
B_{11} & B_{12}\\
B_{12}^{\mathsf T} & B_{22}
\end{pmatrix},
\]
where
$B_{11}\in\Sym^2 U^\ast, B_{12}\in U^\ast\otimes R^\ast$ and $B_{22}\in\Sym^2 R^\ast$.
Taking $3\times 3$ minors of the form 
\[
\begin{pmatrix}
I_2+\varepsilon\ast & \varepsilon\ast\\
\varepsilon\ast & \varepsilon b_{ij}
\end{pmatrix}.
\]
we see that if $\mathrm{rk}\left( A+\varepsilon B\right) \le2$ then $B_{22}=0$. On the other hand, assuming that $B_{22}=0$ and $\mathrm{rk} A=2$ a simple computation shows that
all $3\times 3$ minors vanish in
$ \C_\varepsilon$, hence
$A+\varepsilon B\in\mathcal M\left(  \C_\varepsilon\right)$.
We conclude that
\begin{equation}\label{eq:TangentDeterminantal}
T_A\mathcal M = \left\{
B\in\Sym^2 V^\ast \big| 
B_{22}=0 \right\}
=
\left\{ B\in\Sym^2 V^\ast \big| B\left( \ker\left( A \right) \right)
\subseteq \mathrm{Im}\left( A \right)
\right\}.
\end{equation}
\end{rem}
Recall from \cite[Section 5]{DHOST} that, once we fixed a scalar product on the affine space $\C^{N+1}$, the conormal variety $\mathcal N _X$ of $X\su\P^{N}$ with affine cone $\tilde{X}\su \C^{N+1}$ is the projection to $\P^N \times\P^N$ of
\[
\overline{\{\left(x,y\right)\in\C^{N+1} \times\C^{N+1}\,\big| \,x\in \tilde{X}_\mathrm{reg}\, ,y\perp T_x \tilde{X}\} }
\]
Let $G=\Gr_2 \left( V \right)$ be the Grassmannian of $2$-dimensional
subspaces of $V$ and let $\mathcal{U}$ be its tautological bundle.
Define $Z=\P\left( \mathrm{Sym}^2 \mathcal{U} \right)$. The natural inclusion
$\mathrm{Sym}^2 \mathcal{U} \to \mathrm{Sym}^2 V \otimes \mathcal O_G$
induces a morphism $\pi: Z \to \P\left( \mathrm{Sym}^2 V \right)$ whose image is $\mathcal{M}$. Explicitly 
\begin{equation}\label{eq:KempfResolution}
Z=\left\{ \left( \left[ Q \right], H \right) \in \P\left( \mathrm{Sym}^2 V \right) \times\Gr_2 \left( V \right)  \big| \left[ Q \right] \in \P\left( \mathrm{Sym}^2 H \right)\right\}.
\end{equation}
It is a classical result (see \cite[Chapter 6]{Wey}) that $\pi$ is a
rational resolution of singularities, usually called the Kempf's resolution of $\mathcal M$. In particular $R^{i}\pi_{\ast}\mathcal{O}_{Z}=0$ for all $i>0$.

\section{Neurovarieties with scalar output and the $\ED$-degree}
\label{sec:neurovarieties}

We introduce here the central notions of the paper. The definitions and results in this section are stated in a slightly greater generality than what will be strictly needed later on. This choice is motivated by the fact that we think such a level of generality provides a clearer geometric picture and helps place our main result in the broader context of Euclidean distance degree theory. Further details can be found in \cite{KTB19, KMMT22, Sha24}. See also \cite{KLW24} for a comprehensive introduction to polynomial neural networks and numerous examples.

\subsection{Definitions}
\label{section:Definitions}

\begin{defn}
Let $X$ be a real vector space, $k,r\ge1$ be integers and set
\[
\mathcal{P}_{k} = \Hom\left( X,\R^{k} \right) \oplus
\Hom\left( \R^{k} , \R \right).
\]
An algebraic neural network with one hidden layer of width $k$,
degree $r$ and scalar output is the assignment
\begin{align*}
\Lambda:\mathcal{P}_{k} &\longrightarrow \mathcal{C}\left( X,\R \right)\\
\theta=\left( \phi_{1},\phi_{2} \right) &\longmapsto f_{\theta} = \phi_{2}\circ\underline{\rho}\circ\phi_{1}
\end{align*}
where
\[
\underline{\rho}
\left(
t_{1}{\mathbf e}_{1}+\dots+t_{k}{\mathbf e}_{k}
\right)
=
t_{1}^{r}{\mathbf e}_{1}+\dots+t_{k}^{r}{\mathbf e}_{k}
\]
and $\mathcal{C}$ denotes the space of continuous functions. The functional space of the neural network, denoted by $\mathcal{N}_{k,r}$, is the image of the parameter space under this map:
\[
\mathcal{N}_{k,r} = \im\left( \Lambda \right) \subseteq \mathcal{C}\left( X,\R \right).
\]
\end{defn}

\begin{rem}
It is important to work with fixed basis elements in the hidden layer
since the map $f_{\theta}$ depends on such choice. These basis elements
are called the neurons of the network. The functions in
$\mathcal{N}_{k,r}$ can be described as follows: there exist linear
forms $L_{1},\dots,L_{k}$ on $X$ and scalars $y_{1},\dots,y_{k}\in \R$
such that
\begin{equation}
f_{\theta}\left( x \right)
=
\sum_{i=1}^{k} y_{i} L_{i}\left( x \right)^{r}.
\label{eq:f_thetaExplicit}
\end{equation}
In particular we can identify $f_{\theta}$ with an element
\begin{equation}
f_{\theta}
\in
\mathrm{Sym}^{r}\left( X^{\ast} \right)
\quad\text{with symmetric rank at most } k.
\label{eq:PartiallySymmetricTensor}
\end{equation}
\end{rem}

\begin{defn}
Let
\[
\mathcal{P}_{k} = \Hom\left( X,\R^{k} \right) \oplus \Hom\left( \R^{k} , \R \right)
\]
be the parameter space. We define the neuroalgebraic variety (or neural
variety) $\mathcal{M}_{k,r}$ as the Zariski closure of the image of the
parametrization map inside the projective space of tensors
\[
\mathcal{M}_{k,r} := \overline{ \P\left( \im\left( \Lambda \right) \right) } \subseteq \P\left(  \mathrm{Sym}^{r}\left( X^{\ast} \right) \right).
\]
\end{defn}
In order to compactify the parameter space we decompose it as a sum
of parameter spaces one for each of the $k$ neurons:
\[
\mathcal{P}_{k} = \Hom\left( X,\R^{k} \right) \oplus \Hom\left( \R^{k} , \R \right) \simeq \bigoplus_{i=1}^{k} \left( X^{\ast} \op \R \right)
\]
and take the projectivization accordingly.

\begin{rem}
    Even if we will not strictly need it, an important role in the theory is played by the compactification of the parameter space
\[
\mathcal{W}_{k} = \P\left( X^{\ast} \op \R \right)\times\dots\times\P\left( X^{\ast} \op \R \right)
= \P\left( V_{1} \right)\times\dots\times\P\left( V_{k} \right)
\]
where the notation $\P\left( V_{i} \right)$ is used to emphasize that the space is the space of parameters of the $i$-th neuron. In fact in view of \eqref{eq:f_thetaExplicit} we see that we have a rational dominant map
\[
\Lambda:\mathcal{W}_{k}\dasharrow\mathcal{M}_{k,r},
\]
hence making the variety $\mathcal{M}_{k,r}$ unirational. This is the starting point of a range of numerical approaches that use $\Lambda$ to pull back equations from $\mathcal M _{k,r}$ to $\mathcal W _k$ which is just a product of linear spaces. The main problems arise in view of the high ramification of $\Lambda$ such as the action of $S_k$ on the fibres given by interchanging neurons. See \cite{DHOST} for a general introduction to this idea and \cite{BFHP25} for a concrete example of such application.
\end{rem}

In view of \eqref{eq:PartiallySymmetricTensor}, it is clear that $\mathcal{M}_{k,r}$ naturally identifies with the $k$-th secant variety of the Veronese
\begin{equation}
\sigma_{k}\left(v_{r}\left( X^{\ast} \right)\right)\su\P\left(\mathrm{Sym}^{r}\left( X^{\ast} \right)\right).
\end{equation}

\subsection{Norms and scalar products}\label{sec:norms}

In this section we define the norms and the scalar products that we
will use to define the loss functions. Motivated by the minimality results in \cite{KMQS25} we focus on the Bombieri-Weyl construction of the scalar product on the space of symmetric tensors. Following \cite[Chapter 6]{BFHP25}
let $\underline{\ga}=\left( \ga_{1},\dots,\ga_{n} \right)$ be a multi-index
with $\left| \underline{\ga} \right|=\ga_{1}+\dots+\ga_{n}=d$, then
we set
\[
\binom{d}{\underline{\ga}}
=
\frac{d!}{\ga_{1}!\dots\ga_{n}!}.
\]
Let now $V$ be a scalar product space and fix an orthonormal basis
$\left\{ e_{1},\dots,e_{n} \right\}$ of $V$. Any symmetric tensor
$f\in\mathrm{Sym}^{d}\left( V \right)$ can be uniquely written as:
\[
f
=
\sum_{\left| \underline{\alpha} \right|=d}
\binom{d}{\underline{\alpha}}
f_{\underline{\alpha}}
e^{\underline{\alpha}}.
\]
where $e^{\underline{\alpha}}$ denotes the symmetrized tensor product\footnote{With this we mean 
\[
e_{1}^{\alpha_{1}}\dots e_{n}^{\alpha_{n}} = \frac{1}{d!} \sum \left( \mbox{all permutations of the \ensuremath{d} tensors
\ensuremath{e_{1}^{\ot\ga_{1}}\ot\dots\ot e_{n}^{\ot\ga_{n}}}}
\right)
\]
}
$e_{1}^{\alpha_{1}}\dots e_{n}^{\alpha_{n}}$ and
$f_{\underline{\alpha}}\in\mathbb{R}$.

\begin{defn}
Let $f,g\in\mathrm{Sym}^{d}\left( V \right)$, we define their
Bombieri-Weyl product (also called BW-product)
\[
\left\langle f,g \right\rangle
=
\left\langle f,g \right\rangle_{\mathrm{BW}}
=
\sum_{\left| \underline{\ga} \right|=d}
\binom{d}{\underline{\ga}}
f_{\underline{\ga}} g_{\underline{\ga}}.
\]
\end{defn}

\begin{lem}
\label{lem:BWInnerProductIndependent}
The Bombieri-Weyl product is a scalar product, that is, it is
symmetric, linear in both variables, and positive-definite.
Moreover, it does not depend on the choice of the orthonormal basis.
\end{lem}

\begin{proof}
We include the proof for lack of reference. The first part is clear. To see that the Bombieri-Weyl scalar product
is invariant under orthogonal changes of basis let $Q$ be an
orthogonal transformation on $V$ and let
$f\in\mathrm{Sym}^{d}\left( V \right)$,
let moreover $n=\dim V$. Denote by $f\circ Q$ the polynomial defined
by
\[
\left( f\circ Q \right)\left( x \right)
=
f\left( Qx \right).
\]
Since the Veronese map is non-degenerate the set of powers of linear
forms $\left\{ v^{d}\mid v\in V \right\}$ spans the space
$\mathrm{Sym}^{d}\left( V \right)$. By the linearity of the scalar
product, it suffices to prove the invariance for elements of this
form. For $v\in V$ we identify it with the linear map it defines via
duality
\[
v\left( x \right)
=
\left\langle v,x \right\rangle_{V},
\]
hence the symmetric power is the polynomial
\[
v^{d}\left( x \right)
=
\left\langle v,x \right\rangle_{V}^{d}.
\]
Let $v,w\in V$ then applying the definition of the product and the
Multinomial Theorem:
\begin{equation}
\left\langle v^{d},w^{d} \right\rangle_{\mathrm{BW}}
=
\sum_{\left| \alpha \right|=d}
\binom{d}{\alpha}
v^{\alpha} w^{\alpha}
=
\left( \sum_{i=1}^{n} v_{i} w_{i} \right)^{d}
=
\left\langle v,w \right\rangle_{V}^{d}.
\label{eq:CompatibilityBWInnerProduct}
\end{equation}
Now consider the action of the orthogonal transformation $Q$. The
polynomial transforms as
\[
\left( v^{d} \right)\circ Q
=
\left( Q^{T}v \right)^{d}.
\]
Therefore:
\[
\left\langle
\left( v^{d} \right)\circ Q,
\left( w^{d} \right)\circ Q
\right\rangle_{\mathrm{BW}}
=
\left\langle
\left( Q^{T}v \right)^{d},
\left( Q^{T}w \right)^{d}
\right\rangle_{\mathrm{BW}}
=
\left\langle Q^{T}v,Q^{T}w \right\rangle_{V}^{d}.
\]
Since $Q$ is orthogonal,
\[
\left\langle Q^{T}v,Q^{T}w \right\rangle_{V}
=
\left\langle v,w \right\rangle_{V}.
\]
Thus
\[
\left\langle
\left( Q^{T}v \right)^{d},
\left( Q^{T}w \right)^{d}
\right\rangle_{\mathrm{BW}}
=
\left\langle v^{d},w^{d} \right\rangle_{\mathrm{BW}}.
\]
This proves the invariance.
\end{proof}

Let now $X,Y$ be finite dimensional vector spaces and suppose they have
scalar products
$\left\langle \bullet \right\rangle_{X}$ and
$\left\langle \bullet \right\rangle_{Y}$ and consider the natural
product
$\left\langle \bullet \right\rangle_{X^{\ast}}$ on $X^{\ast}$.
In view of the above construction we can define a natural scalar product
on
$Y\ot \mathrm{Sym}^{r}\left( X^{\ast} \right)$
taking
\[
\left\langle
y\ot\underline{x},
y'\ot\underline{x}'
\right\rangle
=
\left\langle y,y' \right\rangle_{Y}
\cdot
\left\langle \underline{x},\underline{x}' \right\rangle_{\mathrm{BW}}
\]
and extending it by linearity. For an element
$f\in Y\ot \mathrm{Sym}^{r}\left( X^{\ast} \right)$
its norm is defined in the usual way
\[
\left\Vert f \right\Vert^{2}
=
\left\langle f,f \right\rangle.
\]

A fundamental notion in the study of neuroalgebraic varieties is that of isotropic quadric associated with a scalar product: since algebraic objects arise as complexification of their counterparts defined over $\R$, with scalar product we mean the complex-bilinear extension of a real positive definite inner product. In particular, hermitian forms do not appear in this context.
\begin{defn}
    Let $q$ be a non degenerate scalar product on $\C^{n+1}$, the isotropic quadric associated with $q$ is the hypersurface $Q\su\P^n$ defined as the zero locus of the quadratic form associated with $q$.
\end{defn}

For a fixed choice of a basis there is a one-to-one correspondence between smooth quadric hypersurfaces $Q\su\P^n$ with no real points and nondegenerate scalar products on $\C^{n+1}$ up to multiplication by a scalar.

\begin{example}\label{exa:BombieriWeyl}
    After fixing a basis $e_1,\dots, e_n$ of $\C^n$ we identify $A\in\Sym^2 \C^n$ with its sequence of entries $\left(a_{ij}\right)$, then the isotropic quadric associated with the Bombieri-Weyl product is
    \[
    Q_\BW=\left( \sum_{i=1} ^{n} x_{ii} ^2 + 2\sum_{i<j}x_{ij}^2 =0 \right)\su\P\left( \Sym^2 \C^n\right)
    \]
\end{example}

\subsection{The $\protect\ED$-degree and the $\protect\ED$-correspondence}\label{sec:EDGeneral}

For this section fix a nondegenerate scalar product $q$ on $\C^{N+1}$ with isotropic quadric $Q\su\P^N$. We will refer to $q$ and $Q$ interchangeably.

\begin{defn}
Let $X\su\R^{N}$ be a reduced and irreducible algebraic variety with a smooth real point, denote with $X_{\C}\su\C^{N}$ its complexification and let $u\in\C^{N}$. We say that $x\in X_{\C,{\mathrm{reg}}}$ is $Q$-critical with respect to $u$, or just critical if no ambiguity can arise, if $x-u$ is orthogonal to $T_{x}X_{\C}$ with respect to $Q$.
\end{defn}

It is clear that this definition makes sense only for $u\not\in X_{\C}$. Following \cite[Theorem 4.1]{DHOST} and the subsequent discussion let 
\[
\mathcal{E}_{X,Q}=\left\{ \left(x,u\right)\in X_{\C,{\mathrm{reg}}}\times\C^{N}\,\big| \,\mbox{\ensuremath{x} is critical wrt \ensuremath{u}}\right\} 
\]
then the first projection $\pi_{1}:\mathcal{E}_{X,Q}\to X_{\C,{\mathrm{reg}}}$
identifies the fibre $\pi_{1}^{-1}\left(x\right)=x+N_{x}X$ with the
fibre of the affine normal bundle, hence $\pi_{1}$ is an affine bundle
on $X_{\C,{\mathrm{reg}}}$ of rank $c$, hence $\mathcal{E}_{X,Q}$ has
dimension $N$, moreover the second projection $\pi_{2}:\mathcal{E}_{X,Q}\to\C^{N}$
is then dominant between varieties of the same dimension\footnote{This is a consequence of the existence of a smooth real point, see \cite[Theorem 4.1]{DHOST}.}: it is generically finite and the fibres have the same cardinality
for generic $u\in\C^{N}$. Therefore it makes sense to define 
\begin{defn}\label{def:EDDegree}
With setting and notations as before, we define $\ED_Q\left(X\right)$
as the cardinality of the generic fibre of the projection $\pi_{2}:\mathcal{E}_{X,Q}\to\C^{N}$,
called the $\ED$-degree of $X$ with respect to $Q$. If $X\su\P\left(V\right)$ is a
reduced and irreducible real projective variety we define $\ED_Q\left(X\right)$
as the $\ED$-degree of its affine cone $C\left(X_{\C}\right)\su V\ot\C$ with respect to $Q$. The bundle $\pi=\pi_{1}:\mathcal{E}_{X,Q}\to X_{\C,\mathrm{reg}}$ is
called the $\ED$-correspondence of $X$ with respect to $Q$. We will denote the $\ED$-degree of the $k$-th secant variety of the Segre product $\Seg\left(\P\left( \R^m\right) \times v_r \left(\R^{m,\ast}\right) \right)$ with respect to $Q$ as
\[
\ED_Q\left( m,n,k,r\right).
\]
\end{defn}

\begin{rem}\label{rem:ED}
Some Remarks are in order:
\begin{enumerate}
\item The notion of criticality of a point depends only on the projective class of the scalar product, therefore it makes sense to refer to the isotropic quadric in the definition even if $\ED_Q$ is computed using the affine cone;
\item in view of \cite{DHOST} the number $\ED_Q\left(X\right)$ counts the
number of complex critical points of the distance function
\[
d_{u}\left(x\right)=\left\Vert x-u\right\Vert ^{2}:X_{\C}\to\C
\]
for a general $u\in\C^{N}\backslash X_{\C,{\mathrm{reg}}}$. The number
of such real points is usually smaller.
\item \label{item:isotropicQuadric} the definition in the case of a projective variety has some complication which will not affect us in this paper: one can show (see \cite[Lemma 2.8 and Corollary 2.9]{DHOST}) that this definition if well posed in case $X$ is not contained in the isotropic quadric. This is always the case for us since our varieties are real.
\item While the variety $X$ is defined over $\R$, the bundle $\mathcal{E}_{X,Q}$
is defined over the complexification. If $X\su\P\left(V\right)$ is
projective then the map $V\backslash\left\{ 0\right\} \to\P\left(V\right)$
induced a vector bundle over $X_{\C,\mathrm{reg}}$ which by abuse
of notations, we still denote with
\[
\mathcal{E}_{X,Q}\to X_{\C,\mathrm{reg}}.
\]
When the variety $X$ is smooth, for example in the case when $X=\mathcal{M}_{k,r}$
then $\mathcal{E}_{X,Q}$ is the total space of the normal bundle of
the embedding.
\end{enumerate}
\end{rem}

\subsection{The virtual $\ED$-degree}\label{sec:vED}

While the $\ED$-degree introduced in Section \ref{sec:EDGeneral} is a metric invariant of the variety, we introduce here a purely projective analogue. We fix a nondegenerate scalar product $q$ with isotropic quadric $Q$.

\begin{defn}
Let $X\su \P^N$ be a reduced and irreducible variety, then the conormal variety of $X$ is defined as 
\[
\mathcal N _X = \overline{\{ \left(x,H\right)\in \P^N \times \left( \P^N\right)^\ast\,\big| \, x\in X_\mathrm{reg}\, , T_x X \su H\} }
\]
where the closure is taken in $\mathbb P^N \times (\mathbb P^N)^\ast$. It comes with the two projections
\[
\pi_X : \mathcal{N}_X \to \P^N\quad\quad \pi_{ X^\ast }: \mathcal{N}_X \to\left( \P^N\right)^\ast.
\]
\end{defn}

\begin{rem}\label{rem:RelationsEDcorrespondenceConormal}
    The nondegenerate scalar product on $\C^{N+1}$ determines an identification $\P^N \simeq \left( \P^N \right)^\ast$ and it is then possible to see $X$ and $X^\ast$ inside the same projective space via
    \[
\mathcal N _X = \overline{\{ \left(x,y\right)\in \P^N \times \P^N\,\big| \, x\in X_\mathrm{reg}\, , y\perp _q T_x X \} }.
\]
This is the point of view adopted in \cite{DHOST}. In particular, if we denote with $\P\left( \mathcal E ^{(u)} _{X,Q}\right)$ the projectivization of the $\ED$-correspondence for a generic $u\in\C^{N+1}$ we have
\[
\P\left( \mathcal E ^{(u)} _{X,Q}\right) =\{ \left( x,[y] \right)\in X_\mathrm{reg} \times \P^N \,\big| \,y\perp _q T_x X\, ,\,y\in\langle x-u\rangle \}.
\]
It follows that, if $\Gamma _{u,Q} \su \P^N \times\left( \P^N \right)^\ast$ is the closure of the graph of $x\mapsto H_u (x)$ where $H_u (x)$ is the hyperplane corresponding to $x-u$,
\[
\P\left( \mathcal E ^{(u)} _{X,Q}\right) = \mathcal N _X\cap\Gamma_{u,Q}
\]
where we used the scalar product to identify $\P^N \simeq\left( \P^N \right)^\ast$. Note that $\Gamma _{u,Q}$ is rationally equivalent to the diagonal $\Delta\su \P^N \times \P^N$.
\end{rem}

It is clear that $\pi_X$ defines a morphism $\pi_X : \mathcal{N}_X \to X$ and that the restriction 
\[
\mathcal{N}_ {X_\mathrm{reg}} = \pi _X ^{-1} \left( X_\mathrm{reg}\right) \to X_\mathrm{reg}
\]
is the projectivized conormal bundle parametrizing hyperplanes tangent to $X$. We call the image  $\pi_{ X^\ast }\left( \mathcal{N}_X\right)\su \left( \P^N\right)^\ast$ the dual variety of $X$, denoted $X^\ast$. If $\dim X =d$ then for a smooth point $x\in X$ the fibre $\pi^{-1} _X (x)$ is the projective space of hyperplanes containing $T_x X$ hence has dimension $N-d-1$. It follows that $\mathcal N_X$ is an irreducible variety of dimension $N-1$. Its conormal cycle, which by an abuse of notation we still denote with $\mathcal{N} _X$, admits a multidegree decomposition in the Chow ring
\[
\left[ \mathcal{N} _X \right]\in A
^\bullet\left( \P^N \times\left(\P^N \right) ^\ast\right)\simeq\frac{\mathbb Z \left[x,y\right]}{\left( x^{N+1},y^{N+1}\right)}
\]
where $x,y$ denote the hyperplane classes of the two factors and it makes sense to give the following definition.
\begin{defn}
    With setting and notations as above define $\delta_i \left( X\right)\in\mathbb Z$ as
    \[
    \left[ \mathcal N _X \right] = \sum _{i=0} ^{d} \delta_i \left( X\right)x^{d-i} y^{N-1 -d+i}.
    \]
    The numbers $\delta_i \left( X\right)$ are called the polar degrees of $X$ and we define the virtual Euclidean Distance degree of $X$ as
    \[
    \vED \left(X\right)=\sum _{i=0} ^{d} \delta_i \left( X\right).
    \]
    We will denote the $\vED$-degree of the $k$-th secant variety of the Segre product $\Seg\left(\P\left( \R^m\right) \times v_r \left(\R^{m,\ast}\right) \right)$ as
\[
\vED\left( m,n,k,r\right).
\]
\end{defn}

\begin{prop}\label{prop:VirtualEDInequality}
Let $X \subset \mathbb{P}^{N}$ be a reduced and irreducible variety of dimension $d$.
Fix a nondegenerate scalar product on $\C^{N+1}$ with isotropic quadric $Q$, then $$\ED_Q(X)\le \vED(X).$$
\end{prop}

\begin{proof}
In $A^{N}(\mathbb P^N\times\mathbb P^N)$ the diagonal has class $[\Delta] = \sum_{j=0}^{N} x^{j}y^{N-j}$ hence
\[
[\mathcal N_X]\cdot[\Delta]\cdot x
=
\left(\sum_{i=0}^{d} \delta_i(X)\, x^{d-i} y^{N-1-d+i}\right)
\left(\sum_{j=0}^{N} x^{j}y^{N-j}\right)x.
\]
The only monomials contributing to the coefficient of $x^{N}y^{N}$ are those for which
\[
(d-i)+j+1 = N
\qquad\text{and}\qquad
(N-1-d+i)+(N-j)=N,
\]
equivalently $j=N-1-d+i$. For each $i=0,\dots,d$ there is exactly one such $j$, hence
\[
\deg\bigl([\mathcal N_X]\cdot[\Delta]\cdot x\bigr)
=
\sum_{i=0}^{d}\delta_i(X)
= \vED(X).
\]
In view of Remark \ref{rem:RelationsEDcorrespondenceConormal} we see that then $\vED(X)=\deg\bigl([\mathcal N_X]\cdot[\Gamma _u]\cdot x\bigr)$, but the $\ED$-degree counts only those points counted by $\deg\bigl([\mathcal N_X]\cdot[\Gamma _u]\cdot x\bigr)$ that are also affine critical points in $X_{\mathbb C,\mathrm{reg}}$, hence  
$$\ED_Q(X)\le \deg\bigl( [\mathcal N_X]\cdot [\Gamma_u] \cdot x\bigr) =\deg\bigl( [\mathcal N_X]\cdot [\Delta] \cdot x \bigr)= \vED(X).$$
\end{proof}

\begin{rem}\label{rem:Diagonal}
The inequality in Proposition \ref{prop:VirtualEDInequality} can be strict, see for example  \cite[Theorem 6.11]{DHOST} and the subsequent examples \cite[Examples 6.12, 6.13]{DHOST}. Sufficient conditions for the inequality to be an equality are that the isotropic quadric associated with the scalar product has transverse intersection with $X$ and it is disjoint from $X_\mathrm{sing}$ or, more generally, that $\mathcal N _X$ does not intersect the diagonal $\Delta\su\P^N \times\P^N$ set theoretically (see \cite[Theorem 5.4]{DHOST}).
\end{rem}

Although well-known, we include a proof of the following Proposition for completeness, as we could not find a clean reference in the singular setting.

\begin{prop}\label{prop:generalScalarProduct}
Let $X \subset \mathbb{P}^{N}$ be a reduced and irreducible variety. There exists a nonempty Zariski open subset $U$ in the space of nondegenerate symmetric forms on $\mathbb{C}^{N+1}$ such that for any scalar product $q \in U$, the associated conormal variety $\mathcal{N}_{X}$ does not intersect the diagonal $\Delta_{q} \subset \mathbb{P}^{N} \times \mathbb{P}^{N}$ induced by the identification $\mathbb{P}^{N} \simeq \left( \mathbb{P}^{N} \right)^{\ast}$ via $q$.
Consequently, for a generic scalar product, $\ED_Q\left( X \right) = \vED\left( X \right)$.
\end{prop}

\begin{proof}
We prove it
Let $V = \mathbb{C}^{N+1}$ and let $\mathcal{S} = \mathbb{P}\left( \mathrm{Sym}^{2} V^{\ast} \right)$ be the space of symmetric forms defined up to scaling. We consider the incidence variety $\mathcal{Z} \subset \mathcal{N}_{X} \times \mathcal{S}$ defined as:
\[
\mathcal{Z} = \left\{ \left( \xi, \left[ q \right] \right) \in \mathcal{N}_{X} \times \mathcal{S} \mid \xi \in \Delta_{q} \right\},
\]
where $\xi = \left( \left[ x \right], \left[ y \right] \right) \in \mathcal{N}_{X} \subset \mathbb{P}\left( V \right) \times \mathbb{P}\left( V^{\ast} \right)$. It comes with the two projections
\[
\xymatrix{
& \mathcal{Z} \ar[dl]_{\pi_1} \ar[dr]^{\pi_2} \\
\mathcal{N}_{X} & & \mathcal S
}
\]
and we proceed by dimension count: first we use $\pi_1$ to determine $\dim\mathcal Z$ that will be less than $\dim\mathcal S$, hence its image cannot cover all the symmetric forms.

By definition of projective space, representative vectors $x \in V$ and $y \in V^{\ast}$ are non-zero. The condition $\xi \in \Delta_{q}$ means that the polarity induced by $q$ maps $\left[ x \right]$ to $\left[ y \right]$, which in coordinates is equivalent to $qx \in \langle y \rangle$.
Let us consider the first projection $\pi_{1}: \mathcal{Z} \to \mathcal{N}_{X}$. The fiber over a fixed $\xi$ is the linear subspace of forms $q$ satisfying $qx \in \langle y \rangle$.
Consider the evaluation map $\phi_{x}: \mathrm{Sym}^{2} V^{\ast} \to V^{\ast}$ defined by $q \mapsto qx$. Since $x \neq 0$, this map is surjective. Indeed, given $\ell \in V^{\ast}$, we choose a basis where $x = e_{0}$. Define $q$ such that its first row (and column) corresponds to the components of $\ell$ and zero elsewhere yields $qx = \ell$.
The condition $qx \in \langle y \rangle$ is equivalent to saying that the image of $q$ under $\phi_{x}$ lies in the $1$-dimensional subspace spanned by $y$. Consider the quotient map $\pi: V^{\ast} \to V^{\ast}/\langle y \rangle$. Since $y \neq 0$, this is a projection onto a space of dimension $N$. The composition $\psi = \pi \circ \phi_{x}: \mathrm{Sym}^{2} V^{\ast} \to V^{\ast}/\langle y \rangle$ is surjective because both $\phi_{x}$ and $\pi$ are surjective.
The fiber $\pi_{1}^{-1}\left( \xi \right)$ corresponds to $\mathbb{P}\left( \ker \psi \right)$. Since $\psi$ is surjective onto an $N$-dimensional space, the kernel has codimension $N$ in $\mathrm{Sym}^{2} V^{\ast}$. Thus, for every $\xi \in \mathcal{N}_{X}$, the fiber in $\mathcal{S}$ imposes $N$ independent linear conditions.
We compute the dimension of the total space:
\[
\dim \mathcal{Z} = \dim \mathcal{N}_{X} + \dim \pi_{1}^{-1}\left( \xi \right) = \left( N-1 \right) + \left( \dim \mathcal{S} - N \right) = \dim \mathcal{S} - 1.
\]
Consider now the second projection $\pi_{2}: \mathcal{Z} \to \mathcal{S}$. Its image is the locus of forms for which the intersection between $\mathcal N _X$ and the induced diagonal is non-empty. Since $\dim \mathcal{Z} < \dim \mathcal{S}$, the closure $\mathcal{B} = \overline{\pi_{2}\left( \mathcal{Z} \right)}$ is a proper subvariety of $\mathcal{S}$.
Let $D \subset \mathcal{S}$ be the hypersurface of degenerate forms ($\det q = 0$), then the set $U = \mathcal{S} \setminus \left( \mathcal{B} \cup D \right)$ is non-empty and dense and any $\left[ q \right] \in U$, $q$ is nondegenerate with $\mathcal{N}_{X} \cap \Delta_{q} = \emptyset$. In view of Remark \ref{rem:Diagonal} we have $\ED_Q\left( X \right) = \vED\left( X \right)$.
\end{proof}

\section{Stable polynomiality of $\vED$ in the case of $m=1$ and $r=k=2$}\label{sec:Ourcase}

In this section we analyze the function $n \mapsto \vED\left(1,n,2,2\right)$.
The main idea is to exploit the Kempf resolution $\pi: Z \to \mathcal M$
introduced in Section \ref{section:resolution}. This resolution can be realized
as an incidence variety
\[
\xymatrix{
& Z \ar[dl]_p \ar[dr]^\pi \\
\Gr_2 \left(\C^n\right) & & \mathcal M
}
\]
whose properties, like the fact that it coincides with the
Nash blow-up of $\mathcal M$, allow us to reformulate the computation of the
virtual Euclidean Distance degree as an intersection-theoretic problem on $Z$.
Pushing this problem forward to the Grassmannian $\Gr_2 \left(\C^n\right)$, we can then apply
equivariant localization techniques to establish the stable polynomiality of
$\vED\left(1,n,2,2\right)$.

\subsection{Polar degrees, Chern–Mather classes and Nash blow-up of $\mathcal M$}\label{sec:polardegrees}

As seen in Section \ref{sec:vED} the real projective invariant is the $\vED$ of the neurovariety. We recall here the main points in \cite[Section 1]{Piene79} in order to explicit the virtual $\ED$ degree in terms of intersection theory on the variety.
It follows from its definition that the conormal variety $\mathcal N _X$ of $X$ is naturally embedded in $\P^N \times\left(\P^N\right)^\ast$. Following \cite{Piene79} we can see that\footnote{He uses the notation $\mu _k$ for our $\delta_k$}
\begin{equation}\label{eq:delta_k}
\delta_k \left( X\right)=\deg\left( \left[ \mathcal N_X \right]\cdot \left( h^\vee\right)^{k+1}\cdot h^{m-k}\right)
\end{equation}
where $h,h^\vee$ denote the pullbacks of the hyperplane classes from $\P^N$ and $\left(\P^N\right)^\ast$. We call $M_k\left( X\right)=\left[ \mathcal N_X \right]\cdot \left( h^\vee\right)^{k+1}\cdot h^{m-k}$ the $k$-th polar classes of $X$.

\begin{rem}
    If $X$ is smooth they coincide with the degrees of the classical polar varieties defined via generic linear projections.
\end{rem}

Denote with $p_X:\mathrm{Nash}\left(X\right)\to X$ the Nash blow-up of $X$, namely the closed graph of the Gauss map and let $\nu:\tilde T\to \mathrm{Nash}\left(X\right)$ be the Nash bundle, that is the pullback of the tautological subbundle of $\mathrm{Nash}\left(X\right)$. We call the class 
\[
c_M \left(X\right)=p_\ast \left( c\left(\tilde T\right)\cap\left[\mathrm{Nash}\left(X\right)\right]\right)\in A_\bullet\left(X\right)
\]
the Chern-Mather class of $X$ and denote with $c_{i}^\mathrm{Ma}\left(X\right)\in A_{i} \left(X\right)$ its $i$-dimensional component. This is intrinsic of $X$, that is independent on the choice of the embedding (see \cite[Example 4.2.6]{Fulton98}).
\begin{thm}\label{thm:formulaAluffi}
    For a reduced, irreducible and non degenerated $m$-dimensional variety $X\su\P^N$ we have
    \[\vED \left(X\right)=\sum_{j=0}^{m}(-1)^{m+j} \left( 2^{j+1} -1\right)\deg\left(c_j ^\mathrm{Ma} \left(X\right)\cdot H^j\right)\]
    where $H$ is the hyperplane class in $\P^N$.
\end{thm}
\begin{proof}
    This is \cite[Proposition 2.9]{Alu18}.
\end{proof}

We now specialize the previous discussion to the variety
\[
\mathcal M = \sigma_2 \left(v_2(\mathbb P^{n-1})\right)
\subset \mathbb P(\Sym^2 \mathbb C^n),
\]
namely the variety of symmetric matrices of rank at most $2$. 
Recall that $\mathcal M$ has the expected dimension $m=2n-2$. For $G=\Gr _2 (\mathbb C^n)$ denote by $\mathcal U$ the tautological rank-$2$
subbundle and $\mathcal Q$ the universal quotient bundle.

\begin{prop}\label{prop:ZisNashBlowUp}
The Kempf's resolution (\ref{eq:KempfResolution}) $\pi : Z \to \mathcal M$ is the Nash blow-up of $\mathcal M$.
\end{prop}

\begin{proof}
Over the smooth locus $\mathcal M_{\mathrm{reg}}$, a point corresponds to a
quadratic form of rank exactly $2$, hence it determines uniquely its image
$U\subset \mathbb C^n$.
In particular, the Nash map
\[
\gamma_{\mathcal M} ^\mathrm{Nash}: \mathcal M_{\mathrm{reg}} \longrightarrow
\Gr _m (\Sym^2\mathbb C^n)
\]
factors through the Grassmannian $\Gr _2 (\mathbb C^n)$ by associating to a smooth point its support $2$-plane. The space $Z=\mathbb P(\Sym^2\mathcal U)$ parametrizes precisely the pairs
consisting of a point of $\mathcal M_{\mathrm{reg}}$ together with its
projective tangent space. Therefore $Z$ coincides with the closure of the graph of the Gauss map. By definition, this is the Nash blow-up of $\mathcal M$.
\end{proof}

We define the Nash bundle $\widetilde T \to Z$ as the pullback of the
tautological rank-$m$ subbundle on $\Gr_m(\Sym^2\mathbb C^n)$ via the
Nash map $\gamma_{\mathcal M}^{\mathrm{Nash}}$. Let $\xi = c_1(\mathcal O_Z(1))$ denote the first Chern class of the
tautological line bundle on $Z$. 

\begin{prop}\label{eq:ExactSequence}
   With setting and notations as above, there exists a canonical exact sequence
of vector bundles on $Z=\mathbb P\left(\Sym^2\mathcal U^\ast\right)$
\begin{equation}\label{eq:tag{4.2}}
0\longrightarrow \mathcal{O}_Z\left( -1\right)
\longrightarrow
p^\ast\left( \Sym^2\mathcal{U}^\ast\oplus\left( \mathcal{U}^\ast\otimes \mathcal{Q}^\ast\right)\right)
\longrightarrow
\widetilde T
\longrightarrow 0,
\end{equation}
in particular
\begin{equation}\label{eq:tag{4.3}}
c\left(\widetilde T\right) = \frac{ c \left(p^\ast\!\left(
\Sym^2\mathcal U^\ast \oplus (\mathcal U^\ast \otimes \mathcal Q^\ast)
\right)\right)}{c(\mathcal O_Z(-1))}
=
\frac{c\left(p^\ast\!\left(
\Sym^2\mathcal U^\ast \oplus (\mathcal U^\ast \otimes \mathcal Q^\ast)
\right)\right)}{1-\xi}.
\end{equation}
\end{prop}

\begin{proof}
We work over the dual numbers $\C_\varepsilon = \C\left[ \varepsilon\right] / \left(\varepsilon^2\right)$. Let $G=\Gr_2\left( V \right)$ with tautological sequence
\[
0\to \mathcal{U} \to V\otimes \mathcal{O}_G \to \mathcal{Q} \to 0,
\]
and set $p:Z=\P\left( \Sym^2\mathcal{U}^\ast  \right)\rightarrow G$. A point of $Z$ is a pair $\left( U,\left[ A\right]\right)$ with $U\in G$ and
$A\in\Sym^2U^\ast\setminus\left\{ 0\right\}$.
By definition of the projective bundle, $Z$ carries the tautological exact sequence
\begin{equation}\label{eq:tautologicalZ}
0\longrightarrow \mathcal{O}_Z\left( -1\right)
\longrightarrow p^\ast\left( \Sym^2\mathcal{U}^\ast\right)
\longrightarrow \mathcal{Q}_{\mathrm{rel}}
\longrightarrow 0,
\end{equation}
and its fiber at $\left( U,\left[ A\right]\right)$ is $\mathcal{O}_Z\left( -1\right)\big|_{\left( U,\left[ A\right]\right)}
=
\left\langle A\right\rangle
\subset \Sym^2U^\ast$.
Let $C \left( \mathcal{M} \right)\subset \Sym^2V^\ast$ be the affine cone over
$\mathcal{M}$ and $A\in C \left( \mathcal{M} \right)$ with $\mathrm{rk}\left( A\right)=2$. In view of \eqref{eq:TangentDeterminantal} if we set $U=\mathrm{Im}\left( A \right)$ then there is a canonical identification
\begin{equation}\label{eq:tangentSplitting}
T_AC \left( \mathcal{M} \right)\simeq \Sym^2U \oplus \left( U\otimes \left(\frac{V}{U}\right)^\ast\right).
\end{equation}
When passing to the projective variety, tangent directions at a point $\left[A\right]\in \mathcal M\subset \mathbb P\left(\Sym^2V^\ast\right)$ are described in terms of
first-order deformations over the dual numbers $\C_\varepsilon$. By definition, a $\C_\varepsilon$--point of $\mathbb P\left(\Sym^2V^\ast\right)$
corresponds to a rank-one direct summand of $\Sym^2V^\ast\otimes \C_\varepsilon$ whose reduction modulo $\left(\varepsilon\right)$
is $\langle A\rangle$.
Accordingly, two first-order deformations define the same tangent direction at $\left[A\right]$ if and only if they determine the same rank-one
$\C_\varepsilon$-submodule of $\Sym^2V^\ast\otimes \C_\varepsilon$. After fixing a representative whose reduction modulo $\left(\varepsilon\right)$ is
$A$, this equivalence identifies deformations whose infinitesimal parts differ by a multiple of $A$.
In other words, the tangent direction only depends on the class of $B$ in $\Sym^2V^\ast/\langle A\rangle$.
Therefore, the embedded projective tangent space at $\left[A\right]$ is canonically identified with 
\begin{equation}\label{eq:projectiveTangentDualNumbers}
T_{\left[A\right]}\mathcal M \simeq \frac{T_A C\left(\mathcal M\right)}{\langle A\rangle}.
\end{equation}
Varying $\bigl(U,[A]\bigr)$ over $Z$, the description
\eqref{eq:tangentSplitting} globalizes to the vector bundle
\[
p^\ast\!\left(
\Sym^2\mathcal U^\ast \oplus (\mathcal U^\ast \otimes \mathcal Q^\ast)
\right),
\]
whose fiber at $\bigl(U,[A]\bigr)$ identifies with the Zariski tangent space
$T_A C\left(\mathcal M\right)$ of the affine cone.
The line $\langle A\rangle \subset \Sym^2 U^\ast$ globalizes to the
tautological subbundle
$\mathcal O_Z(-1)\subset p^\ast(\Sym^2\mathcal U^\ast)$
via \eqref{eq:tautologicalZ}.
Passing to the projective tangent space as in
\eqref{eq:projectiveTangentDualNumbers}, we obtain a rank--$m$ vector bundle
on $Z$ whose fiber at a smooth point $\bigl([A],U\bigr)\in Z$ identifies with
the embedded projective tangent space $T_{[A]}\mathcal M$.
In view of Proposition \ref{prop:ZisNashBlowUp} this construction realizes $Z$ as the closure of the graph of the Gauss map of $\mathcal M$, and the above quotient bundle is canonically isomorphic, via the Nash map, to the
pullback of the tautological bundle on $\Gr_m\left(\Sym^2 V^\ast\right)$.
By definition, this bundle is the Nash bundle $\widetilde T$, and we obtain the exact sequence
\[
0\longrightarrow \mathcal{O}_Z\left(-1\right)
\longrightarrow
p^\ast\left(
\Sym^2\mathcal U^\ast \oplus \left(\mathcal U^\ast \otimes \mathcal Q^\ast\right)
\right)
\longrightarrow
\widetilde T
\longrightarrow 0.
\]
\end{proof}


\subsection{Reduction to integrals on $\Gr_2 \left(\C^n\right)$}\label{sec:reductionGr}

We now explain how the computation of polar degrees of $\mathcal M$
can be reduced to intersection-theoretic calculations on the Grassmannian
$G=\Gr _2 (\mathbb C^n)$. Recall that
\[
c_M\left(\mathcal M\right) = \pi_\ast \left( c\left(\widetilde T\right)\cap[Z] \right),
\]
where, in view of Proposition \ref{prop:ZisNashBlowUp} $Z=\mathbb P\left(\Sym^2\mathcal U^\ast\right)$ is the Nash blow-up of $\mathcal M$
and $\widetilde T$ is the Nash bundle. Let $H$ denote the hyperplane class of $\mathcal M\subset\mathbb P(\Sym^2\mathbb C^n)$, then we have $\pi^\ast H=\xi$, where $\xi=c_1(\mathcal O_Z(1))$.
By the projection formula, for every $i,j$ we obtain
\begin{equation}\label{eq:tag{4.4}}
\deg \left(
c^{\mathrm{Ma}}_{m-i}\left(\mathcal M\right)\cdot H^j
\right)
=
\int_{\mathcal M}
c^{\mathrm{Ma}}_{m-i}\left(\mathcal M\right)\cdot H^j
=
\int_Z
c_i\left(\widetilde T\right)\,\xi^j.
\end{equation}
Therefore, each term appearing in Theorem \ref{thm:formulaAluffi} for the polar degrees can be expressed as an integral over $Z$ of the form
\[
\int_Z c_i\left(\widetilde T\right)\,\xi^j.
\]
Using \ref{eq:tag{4.3}}, the class $c_i\left(\widetilde T\right)$ can be written as a polynomial
in $\xi$ with coefficients in the Chow ring of $G$ involving only
the Chern classes of $\mathcal U^\ast$ and $\mathcal Q^\ast$, hence all integrals reduce to expressions of the form
\[
\int_Z \xi^k \cdot p^\ast (\alpha),
\qquad
\alpha\in A^\ast (G).
\]
To evaluate such integrals, we use the projective bundle formula for
$p:Z=\mathbb P(\Sym^2\mathcal U^\ast)\to G$.
Since $\rk(\Sym^2\mathcal U^\ast)=3$, we have
\[
p_*(\xi^{2+t}) = s_t(\Sym^2\mathcal U^\ast),
\qquad t\ge 0,
\]
where $s_t$ denotes the $t$-th Segre class.
Consequently,
\begin{equation}\label{eq:tag{4.5}}
\int_Z \xi^{2+t}\cdot p^\ast (\alpha)
=
\int_G s_t(\Sym^2\mathcal U^\ast)\cdot \alpha.
\end{equation}

Combining (\ref{eq:tag{4.4}}) and (\ref{eq:tag{4.5}}), we conclude that all polar degrees of $\mathcal M$ are obtained by integrating polynomials in the Chern classes of $\mathcal U^\ast$ and $\mathcal Q^\ast$ over the Grassmannian $G$.

\subsection{Equivariant localization on $\Gr_2 \left(\C^n\right)$ and polynomiality}\label{sec:localization}

We now explain how to evaluate the intersection numbers appearing in the
previous section.
Since all integrals are taken over the Grassmannian
$G=\Gr_2(\mathbb C^n)$, we use equivariant localization with respect to the
standard action of the algebraic torus. We refer to \cite{EG982} for details, or \cite{EG98} for a general introduction.

Let $T=(\mathbb C^\ast)^n$ act diagonally on $\mathbb C^n$ with characters
$t_1,\dots,t_n$.
This induces a natural action on $G$, as well as on the tautological
bundles $\mathcal U$ and $\mathcal Q$.
By functoriality, the classes appearing in~\eqref{eq:tag{4.5}} admit canonical equivariant
lifts to the equivariant Chow ring $A_T^\ast(G)$ \cite[Section 2.4]{EG98}. 
We have
\begin{thm}
    Let $F_T$ be the fraction field of the equivariant Chow ring $R_T$, then for every $\alpha \in A_T^\ast(G) \otimes_{\mathbb{Q}} F_T$, one has
\begin{equation}\label{eq:tag{4.6}}
\int_{G} \alpha = \sum_{p \in G^T}
\frac{\alpha|_p}{c_T(T_p G)},
\end{equation}
where $c_T(T_p G)$ is the equivariant top Chern class of the tangent representation at $p$.
\end{thm}

\begin{proof}
  This is \cite[Section 4, Theorem 2 and Corollary 1]{EG982}.
\end{proof}

\begin{rem}\label{rem:identificationEulerClassTangent}
    We point out that the sum in \ref{eq:tag{4.6}} is over the $T$-fixed points of $G$ hence it is finite and $e_T(T_pG)$ denotes the equivariant Euler class of the tangent space at $p$. Note that since $p$ is an isolated fixed point, the tangent space $T_p X$ decomposes as a direct sum of one-dimensional $T$-representations and its equivariant Euler class is given by the product of the corresponding weights.
\end{rem}
The fixed points of the $T$-action on $G$ are indexed by
$2$-element subsets $$I=\{i,j\}\subset\{1,\dots,n\}.$$
We denote by $p_{ij}$ the corresponding fixed point.
At $p_{ij}$, the fiber of the tautological subbundle $\mathcal U$ has
equivariant Chern roots $t_i$ and $t_j$, while the quotient bundle
$\mathcal Q$ has equivariant Chern roots $t_k$ for $k\neq i,j$. In view of the canonical isomorphism
\[
T_{p_{ij}}G \simeq \Hom(U,Q) \simeq U^\vee \otimes Q
\]
and Remark \ref{rem:identificationEulerClassTangent} we have
\begin{equation}\label{eq:tag{4.7}}
e_T(T_{p_{ij}}G) = \prod_{k\neq i,j}(t_k-t_i)(t_k-t_j).
\end{equation}
In view of \eqref{eq:tag{4.5}}, applying the pushforward
$\pi_\ast : A^\ast(Z) \to A^\ast(G)$ to the class
$c_i\left(\widetilde T\right)\,\xi^j$ in (\ref{eq:tag{4.4}}) gives a class
$\alpha_n \in A^\ast(G)$ such that
\[
\deg \bigl(
c^{\mathrm{Ma}}_{m-i}\left(\mathcal M\right)\cdot H^j
\bigr)
=
\int_Z c_i\left(\widetilde T\right)\,\xi^j
=
\int_G \alpha_n ,
\]
in view of localization formula \eqref{eq:tag{4.6}}, we obtain
\begin{equation}\label{eq:tag{4.8}}
\deg \bigl(
c^{\mathrm{Ma}}_{m-i}\left(\mathcal M\right)\cdot H^j
\bigr)=
\sum_{1\le i<j\le n}
\frac{\alpha_n|_{p_{ij}}}{
\prod_{k\neq i,j}(t_k-t_i)(t_k-t_j)
}.
\end{equation}
Since $\alpha_n$ is constructed from universal polynomials in the Chern
classes of the tautological bundles $\mathcal U$ and $\mathcal Q$, its
restriction to a fixed point depends only on the corresponding equivariant
weights.
In particular, the dependence on the weights $t_k$ with $k\neq i,j$ is
symmetric.
As the left-hand side of \eqref{eq:tag{4.8}} represents  ordinary
(nonequivariant) intersection number and degree, it is independent of the equivariant
parameters, hence the right-hand side simplifies to a numerical value.

\begin{thm}\label{prop:polynomiality-n}
The virtual Euclidean distance degree
\[
\vED(1,n,2,2)
\]
is a stably polynomial function of $n$.
\end{thm}

\begin{proof}
In view of Theorem \ref{thm:formulaAluffi} and formula \eqref{eq:tag{4.8}} we can write $\vED(1,n,2,2)$ as a finite sum of integrals
\[
\int_G \alpha _n
\]
where $\alpha_n\in A^\ast(G)$ is obtained from polynomials
in the Chern classes of the tautological bundles $\mathcal U^\ast$ and $\mathcal Q^\ast$
(and Segre classes of $\Sym^2\mathcal U^\ast$) via the pushforward
$\pi_\ast : A^\ast(Z)\to A^\ast(G)$. Choosing an equivariant lift
$\widetilde{\alpha}_n\in A_T^\ast(G)$ of $\alpha_n$, the localization
formula \eqref{eq:tag{4.6}} gives
\[
\int_{G} \alpha_n
=
\sum_{1\le i<j\le n}
\frac{\widetilde{\alpha}_n|_{p_{ij}}}{
\prod_{k\neq i,j}(t_k-t_i)(t_k-t_j)
}.
\]
Each summand is obtained by evaluating a symmetric polynomial in the
equivariant Chern roots of $\mathcal U$ and $\mathcal Q$ at the fixed point
$p_{ij}$ and hence the dependence on the variables $\{t_k\}_{k\neq i,j}$ is symmetric.
Since the localization sum represents an ordinary intersection number, it does not depend on the equivariant parameters, so we can specialize $t_\ell=\ell$ for $\ell=1,\dots,n$.
After this specialization, each summand becomes a finite linear combination
of symmetric sums over the set $\{1,\dots,n\}\setminus\{i,j\}$.
Such symmetric sums are evaluations of symmetric functions of fixed degree,
and hence depend polynomially on $n$ for $n$ sufficiently large,
see \cite[Chapter 1, Section 2]{MacD95}. As the total number of such sums is finite, we conclude.
\end{proof}

\subsection{A question}\label{sec:question}

While the previous discussion settles the problem for the geometric invariant, the $\vED$, it leaves the question about the stable polynomiality of the metric invariant open along a family of spaces with scalar products. A reasonable definition for such a family is the following:
\begin{defn}
    A compatible family of scalar product spaces is a sequence of triples
    \[
    \left( \C^n, q_n , i_n\right)_{n\in\mathbb N}
    \]
    where $q_n$ is a non degenerate scalar product on $\C^n$ and $i_n :\C^n \to \C^{n+1}$ is a linear embedding such that, denoting with $Q_n \su \P^n$ the isotropic quadric associated with $q_{n+1}$, the induced map $\P^{n-1}\to \P^n$ satisfies
    \[
    Q_n \cap \P^{n-1} = Q_{n-1}.
    \]
\end{defn}

Obvious examples of such objects are diagonal scalar product spaces:
\begin{example}
     A family of scalar product spaces is said to be diagonal if there exists a sequence of strictly positive real numbers $\left(  a_n \right)_{n\ge0}$ such that
     \[
     q_n = \mathrm{diag}\left( a_0 ,\dots, a_{n-1}\right) \quad\quad\mathrm{and}\quad\quad i_n \left( x_0\dots,x_{n-1}\right)=\left( x_0\dots,x_{n-1},0\right).
     \]
     It is clear that diagonal families are compatible.
\end{example}

\begin{quest}\label{quest:polynomiality}
    Given a compatible family of scalar product spaces $\left( \C^n, q_n , i_n\right)_{n\in\mathbb N}$, is it true that the function
    \[n\mapsto \ED_{q_n}\left(1,n,2,2\right)\]
    is stably polynomial?
\end{quest}
Clearly question \ref{quest:polynomiality} is equivalent to the fact that
\[n\mapsto \vED\left(1,n,2,2\right) - \ED_{q_n}\left(1,n,2,2\right)\]
is stably polynomial. Note that, in view of Example \ref{exa:BombieriWeyl}, the Bombieri-Weyl scalar product defines a compatible family and it seems reasonable to expect a "nice" behavior.

\section{Numerical results}\label{sec:numerical}
\subsection{Homotopy continuation}\label{sec:Homotopy}

In this section we illustrate, by means of explicit computations, the dependence of the
Euclidean Distance degree on the choice of the scalar product.
The computations are performed using numerical algebraic geometry techniques,
in particular homotopy continuation as implemented in \texttt{Julia}. Here we briefly recall the geometric idea behind this method. Instead of directly solving a
polynomial system $F(x)=0$, which may be difficult due to its degree or number of variables, one constructs a homotopy
connecting $F$ to a simpler system $G$ whose solutions are known. More precisely, one considers a one-parameter family of systems
\[
H(x,t) = (1-t)G(x) + \gamma t F(x) = 0,
\]
where $\gamma\in\mathbb{C}^\ast$ is a generic complex constant and $t\in[0,1]$.
Starting from the solutions of $G(x)=0$ at $t=0$, the algorithm numerically tracks the
corresponding solution paths as $t$ increases, using predictor-corrector methods
(such as Euler or Runge-Kutta prediction followed by Newton correction), until $t=1$.
For a generic choice of $\gamma$, standard results in algebraic geometry
(e.g. Bertini-type theorems and generic deformation arguments) guarantee that the solution
paths do not encounter singularities or bifurcations for $t\in[0,1)$.
Some paths may diverge as $t\to1$; these correspond to solutions at infinity and are
discarded. The number of paths converging to finite nonsingular solutions at $t=1$
equals the number of isolated complex solutions of the original system $F(x)=0$.
In our setting, this number coincides with the Euclidean Distance degree.

All computations below are performed using the package
\texttt{HomotopyContinuation.jl} \cite{BrTi}, available at
\href{https://www.juliahomotopycontinuation.org/}{https://www.juliahomotopycontinuation.org/}.

\subsection{An explicit example: dependence on the metric}\label{sec:NumericalResults}

We focus on the case $m=1$, $k=r=2$, and $n=3$, so that the associated neurovariety
identifies with the hypersurface
\[
\mathcal{M} = \{ \det(X)=0 \} \subset \Sym^2(\mathbb{C}^3),
\]
consisting of symmetric $3\times3$ matrices of rank at most $2$.
In affine coordinates, $\mathcal{M}$ is defined by a single cubic equation in six variables.
Given a scalar product $q$ on $\Sym^2(\mathbb{C}^3)$ and a generic target point $u$,
the ED-critical points of $\mathcal{M}$ are obtained by solving the Lagrange system
\[
\begin{cases}
\det(X)=0, \\
q(x-u) = \lambda \nabla \det(X),
\end{cases}
\]
where $x$ denotes the vector of affine coordinates of $X$ and $\lambda$ is a Lagrange multiplier.
This system consists of seven polynomial equations in seven unknowns and can be treated
numerically via homotopy continuation.
We performed two sets of computations.
In the first case, we chose a generic scalar product on $\Sym^2(\mathbb{C}^3)$, represented
by a random symmetric invertible matrix.
In the second case, we used the Bombieri-Weyl scalar product, which in the chosen coordinates
corresponds to a fixed diagonal matrix.
For each choice of the metric, several random target points $u$ were tested in order to
check numerical stability.
In the generic case, the homotopy continuation consistently returns $13$ isolated nonsingular
solutions, independently of the random choices involved. In view of Proposition \ref{prop:generalScalarProduct} this strongly indicates that
\[
\vED(1,3,2,2)=13.
\]
In contrast, when the Bombieri-Weyl scalar product is used, the same procedure consistently
returns only $3$ solutions. Hence,
\[
\ED_{\mathrm{BW}}(1,3,2,2)=3,
\]
showing a strict inequality
\[
\ED_{\mathrm{BW}}(1,3,2,2) < \ED_{\mathrm{gen}}(1,3,2,2).
\]
This explicit example illustrates concretely the dependence of the Euclidean Distance degree
on the choice of the scalar product and provides a direct comparison between the generic
value and the value associated with a highly symmetric metric.

The Julia code used for the numerical experiments in this section is available at
\url{https://github.com/GiacGraz/ed-degree-metric-dependence}.

\newpage

\bibliographystyle{plain}
\bibliography{References}

@article{Alu18,
  author  = {Aluffi, P.},
  title   = {Projective duality and a {C}hern-{M}ather involution},
  journal = {Transactions of the American Mathematical Society},
  volume  = {370},
  year    = {2018},
  pages   = {1803-1822}
}

@article{BFHP25,
  author  = {Borovik, V. and Friedman, H. and Hosten, S. and Pfeffer, M.},
  title   = {{Numerical Algebraic Geometry for Energy Computations on Tensor Train Varieties}},
  journal = {arXiv preprint arXiv:2512.06939},
  year    = {2025},
  url     = {https://arxiv.org/abs/2512.06939}
}

@article{BrTi,
  author  = {Breiding, P. and Timme, S.},
  title   = {{HomotopyContinuation.jl: A Package for the Numerical Solution of Systems of Polynomial Equations}},
  journal = {International Congress on Mathematical Software},
  pages   = {458-465},
  year    = {2018},
  organization = {Springer},
  doi     = {10.1007/978-3-319-96418-8_54}
}

@article{DHOST,
  author  = {Draisma, J. and Horobet, E. and Ottaviani, G. and Sturmfels, B. and Thomas, R. R.},
  title   = {{The Euclidean Distance Degree of an Algebraic Variety}},
  journal = {Foundations of Computational Mathematics},
  year    = {2016},
  volume  = {16},
  number  = {1},
  pages   = {99-149},
  doi     = {10.1007/s10208-014-9240-x}
}

@article{EG98,
  author  = {Edidin, D. and Graham, W.},
  title   = {Equivariant intersection theory},
  note    = {With an appendix by {A}ngelo {V}istoli: {T}he {C}how ring of {$\mathcal{M}_2$}},
  journal = {Inventiones Mathematicae},
  volume  = {131},
  year    = {1998},
  pages   = {595-634}
}

@article{EG982,
  author  = {Edidin, D. and Graham, W.},
  title   = {{Localization in {E}quivariant {I}ntersection {T}heory and the {B}ott {R}esidue {F}ormula}},
  journal = {{American Journal of Mathematics}},
  volume  = {{120}},
  number  = {{3}},
  year    = {{1998}},
  pages   = {{619-636}},
  publisher = {{The Johns Hopkins University Press}}
}

@book{Fulton98,
  author    = {Fulton, W.},
  title     = {{Intersection Theory}},
  edition   = {2},
  publisher = {Springer},
  series    = {Ergebnisse der Mathematik und ihrer Grenzgebiete},
  year      = {1998},
  address   = {Berlin},
  isbn      = {978-3-540-62046-9}
}

@article{KLW24,
  title={Geometry of polynomial neural networks},
  author={Kubjas, K. and Li, J. and Wiesmann, M.},
  journal={Algebraic Statistics},
  volume={15},
  number={2},
  pages={295--328},
  year={2024},
  publisher={Mathematical Sciences Publishers}
}

@article{KMMT22,
  author  = {Kohn, K. and Merkh, T. and Mont{\'u}far, G. and Trager, M.},
  title   = {{Geometry of Linear Convolutional Networks}},
  journal = {SIAM Journal on Applied Algebra and Geometry},
  year    = {2022},
  volume  = {6},
  number  = {3},
  pages   = {368-406},
  doi     = {10.1137/21M1441183},
  url     = {https://doi.org/10.1137/21M1441183}
}

@article{KMQS25,
  author  = {Kozhasov, K. and Muniz, A. and Qi, Y. and Sodomaco, L.},
  title   = {{On the Minimal Algebraic Complexity of the Rank-One Approximation Problem for General Inner Products}},
  journal = {Mathematics of Computation},
  note    = {To appear},
  year    = {2023},
  eprint  = {2309.15105},
  archivePrefix = {arXiv},
  primaryClass  = {math.NA}
}

@inproceedings{KTB19,
  author    = {Kileel, J. and Trager, M. and Bruna, J.},
  title     = {{On the Expressive Power of Deep Polynomial Neural Networks}},
  booktitle = {Advances in Neural Information Processing Systems},
  volume    = {32},
  year      = {2019},
  publisher = {Curran Associates, Inc.}
}

@book{Lan12,
  author    = {Landsberg, J. M.},
  title     = {{Tensors: Geometry and Applications}},
  publisher = {American Mathematical Society},
  series    = {Graduate Studies in Mathematics},
  volume    = {128},
  year      = {2012},
  address   = {Providence, RI},
  isbn      = {978-0-8218-6907-9}
}

@book{MacD95,
  author    = {Macdonald, I. G.},
  title     = {Symmetric Functions and Hall Polynomials},
  edition   = {Second},
  series    = {Oxford Mathematical Monographs},
  publisher = {Oxford University Press},
  address   = {Oxford},
  year      = {1995}
}

@article{MSMTK25,
  author  = {Marchetti, G. L. and Shahverdi, V. and Mereta, S. and Trager, M. and Kohn, K.},
  title   = {{An Invitation to Neuroalgebraic Geometry}},
  journal = {arXiv preprint arXiv:2501.18915},
  year    = {2025},
  url     = {https://arxiv.org/abs/2501.18915}
}

@article{Piene79,
  author  = {Piene, R.},
  title   = {{Polar Classes of Singular Varieties}},
  journal = {Annales Scientifiques de l'École Normale Supérieure},
  series  = {4},
  volume  = {11},
  pages   = {247-276},
  year    = {1978}
}

@article{Sha24,
  author  = {Shahverdi, V.},
  title   = {{Algebraic Complexity and Neurovariety of Linear Convolutional Networks}},
  journal = {arXiv preprint arXiv:2401.16613},
  year    = {2024}
}

@article{SM17,
  author  = {Sonoda, S. and Murata, N.},
  title   = {{Neural Network with Unbounded Activation Functions Is a Universal Approximator}},
  journal = {Neural Networks},
  year    = {2017},
  volume  = {83},
  pages   = {18-26},
  doi     = {10.1016/j.neunet.2016.07.004}
}

@article{SM18,
  author  = {Sonoda, S. and Murata, N.},
  title   = {{Universal Approximation Theorem for Deep Neural Networks}},
  journal = {Neural Networks},
  year    = {2018},
  volume  = {108},
  pages   = {90-97},
  doi     = {10.1016/j.neunet.2018.08.013}
}

@book{Wat09,
  author    = {Watanabe, S.},
  title     = {{Algebraic Geometry and Statistical Learning Theory}},
  publisher = {Cambridge University Press},
  year      = {2009},
  series    = {Cambridge Monographs on Applied and Computational Mathematics},
  address   = {Cambridge},
  isbn      = {978-0-521-86867-8}
}

@book{Wey,
  author    = {Weyman, J.},
  title     = {{Cohomology of Vector Bundles and Syzygies}},
  series    = {Cambridge Tracts in Mathematics},
  volume    = {149},
  publisher = {Cambridge University Press},
  year      = {2003},
  address   = {Cambridge},
  isbn      = {978-0-521-62198-3}
}

\end{document}